\documentclass[twoside]{amsart}
\usepackage{mathrsfs}
\usepackage{amsfonts}
\usepackage{amssymb}
\usepackage{amssymb}
\usepackage{amssymb}
\usepackage{amssymb}
\usepackage{amssymb}
\usepackage{amssymb}
\usepackage{amssymb}
\usepackage{amssymb}
\usepackage{amsmath}
\usepackage[all]{xy}

\usepackage{lastpage}

\RequirePackage{amsmath} \RequirePackage{amssymb}
\usepackage{amscd,latexsym,amsthm,amsfonts,amssymb,amsmath,amsxtra}
    \newcommand{\BA}{{\mathbb {A}}} 
    \newcommand{\BC}{{\mathbb {C}}} 
     \newcommand{\BF}{{\mathbb {F}}}

     \newcommand{\BR}{{\mathbb {R}}}

    \newcommand{\CO}{{\mathcal {O}}} \newcommand{\CP}{{\mathcal {P}}}
     
    \newcommand{\CS}{{\mathcal {S}}} 
     
    \newcommand{\CW}{{\mathcal {W}}}

    \newcommand{\RU}{{\mathrm {U}}}

     \newcommand{\bx}{{\bf {x}}}

    \newcommand{\lenth}{{\mathrm {\lenth}}}

     \newcommand{\GL}{{\mathrm{GL}}}
    \newcommand{\Hom}{{\mathrm{Hom}}} 
    \newcommand{\Ind}{{\mathrm{Ind}}}

     \newcommand{\Sp}{{\mathrm{Sp}}}

    \newcommand{\cond}{\mathrm{cond}} 
    \renewcommand{\Re}{{\mathrm{Re}}}

 \newcommand{\Vol}{{\mathrm{Vol}}}

 \newcommand{\SL}{{\mathrm{SL}}}
 \newcommand{\SO}{{\mathrm{SO}}}

\newcommand{\vol}{{\mathrm{vol}}}

 \newcommand{\diag}{{\mathrm{diag}}}

    \newcommand{\wpair}[1]{\left\{{#1}\right\}}

     \newcommand{\ra}{\rightarrow}

    \newcommand{\nequiv}{\equiv\hspace{-7.8pt}/}
    \theoremstyle{plain}

    \newtheorem{thm}{Theorem}[section] \newtheorem{cor}[thm]{Corollary}
    \newtheorem{lem}[thm]{Lemma}  \newtheorem{prop}[thm]{Proposition}

    \numberwithin{equation}{section}

\usepackage[top=1in, bottom=1in, left=1.25in, right=1.25in]{geometry}

\title{A strong multiplicity one theorem for $\SL_2$}
\author{Jingsong Chai and Qing Zhang}

\begin{document}

\begin{abstract}
It is known that multiplicity one property holds for $\SL_2$, while
the strong multiplicity one property fails. However, in this paper,
we show that if we require further that a pair of cuspidal
representations $\pi$ and $\pi'$ of $\SL_2$ have the same local
components at archimedean places and the places above 2, and they are generic with respect to the same
additive character, then they also satisfy the strong multiplicity
one property. The proof is based on a local converse theorem for
$\SL_2$.
\end{abstract}

\maketitle



%\tableofcontents

\section{Introduction}

Let $F$ be a number field, and $\BA=\BA_F$ be its ring of adeles.
Let $G$ be a linear reductive algebraic group defined over $F$. The
study of space of automorphic forms $L^2(G(F)\backslash G(\BA))$ has
been a central topic in Langlands program and representation theory.
Let $L^2_0(G(F)\backslash G_(\BA))$ be the subspace of cuspidal
representations.  Suppose $\pi$ is an irreducible automorphic
representation of $G(\BA)$. It is known that $\pi$ occurs discretely
with finite multiplicity $m_\pi$ in $L^2_0(G(F)\backslash G(\BA))$.

The multiplicities $m_\pi$ are important in the study of automorphic
forms and number theory. By the work of Piatetski-Shapiro and
Jacquet and Shalika (\cite{JaSh}) the group $G=\GL_n$ has property
of multiplicity one, that is, $m_\pi\le 1$ for any $\pi$. This is
also true for $\SL_2$ by the famous work of D.Ramakrishnan
(\cite{Ra}). But in general, the multiplicity one property fails,
for example \cite{B,GaGJ,Li,LL} to list a few.

In the case of $\GL_n$, a stronger theorem, called the strong
multiplicity one, holds. It says that two cuspidal representations
$\pi_1,\pi_2$, if they have isomorphic local components almost
everywhere, then they coincide in the space of cusp forms(not only
isomorphic). It follows from the results in \cite{LL} that $\SL_2$
doesn't have this strong multiplicity one property. Multiplicity one
property is already very few, and the strong multiplicity one is
even rare. To the authors' knowledge, the only example other than
$\GL_n$ in this direction is the rigidity theorem for $\SO(2n+1)$,
Theorem 5.3 of \cite{JS}, which is established by D.Jiang and
D.Soudry.

The main purpose of this paper is to prove a weaker version of
strong multiplicity one result for $\Sp_2=\SL_2$. Although we know
the strong multiplicity one doesn't hold in general for a pair of
cuspidal representations $\pi_1,\pi_2$ of $\SL_2(\BA)$, but if we
require that both $\pi_1,\pi_2$ are generic with respect to the same
additive character $\psi$ of $\BA$, then we can show that they also
satisfy the strong multiplicity one property.

The reason for the failure of the strong multiplicity one for
$\SL_2$ is the existence of L-packets. According the local
conjecture of Gan-Gross-Prasad, Conjecture 17.3 of \cite{GaGP},
there is at most one $\psi$-generic representation in each L-packet.
For $\SL_2$, the result is known by the local discussion in
\cite{LL}. In this paper, we prove a local converse theorem for
$\SL_2(F)$ when $F$ is a $p$-adic field such that its residue
characteristic is not 2, which will reprove this result and confirm
a local converse conjecture in \cite{Jng}. This also implies our
version of strong multiplicity one easily.

We now give some details of our results. In
\cite{GPS2}, Gelbart and Piatetski-Shapiro constructed some
Rankin-Selberg integrals to study $L$-functions on the group
$G_n\times \GL(n)$, at least when $G_n$ has split rank $n$. In
particular, in Method C in that paper, if $\pi$ is a globally
generic cuspidal representation of $\Sp_{2n}(\BA)$, $\tau$ is
cuspidal representation of $\GL_n(\BA)$, consider the global Shimura type
zeta-integral
\[
I(s, \phi, E)=\int_{\Sp_{2n}(F)\setminus \Sp_{2n}(\BA)} \phi(g)\theta(g)E(g,s)
\]
where $\phi$ belongs to the space of $\pi$, $E(g,s)$ is a genuine
Eisenstein series on $\widetilde \Sp_{2n}(\BA)$ built from the
representation induced from $\GL_n(\BA)$ by $\tau$ twisted by
$|\det|^s$, and $\theta(g)$ is some theta series on $\widetilde
\Sp_{2n}(\BA)$. Note that the product $\theta(g)E(g,s)$ is
well-defined on $\Sp_{2n}$. The global integral is shown to be
Eulerian. The functional equations and unramified calculations were
also carried out in \cite{GPS2} by Gelbart and Piatetski-Shapiro.
Although we will only consider the easiest case when $n=1$ of
Gelbart and Piatetski-Shapiro's construction, we remark here that
Ginzburg, Rallis and Soudry generalized the above construction to
$\Sp_{2n}\times \GL_k$, for any $k$, in \cite{GiRS}.

We study in more details of Gelbart and Piatetski-Shapiro's local integral
\[
\Psi(W_v, \phi_v, f_{s,v})=\int_{N(F_v)\setminus
\SL_2(F_v)}W_v(h)(\omega_{\psi_v^{-1}}(h)\phi_v)(1)f_{s,v}(h)dh
\]
(for the unexplained notations, see sections below) when $v$ is finite. These local zeta integrals satisfy
certain functional equations, which come from the intertwining
operators on induced representation and certain uniqueness
statements. These functional equations can then be used to define
local gamma factors $\gamma(s,\pi_v,\eta_v,\psi_v)$, where $\pi_v$
is a generic representation of $\SL_2(F_v)$, $\eta_v$ is a character
of $F_v^\times$, and $\psi_v$ is a nontrivial additive character.
The main local result of this paper can be formulated as follows.
\begin{thm}[Theorem $3.10$, local converse theorem and stability of
$\gamma$] Suppose that the residue characteristic of a p-adic field
$F$ is not $2$ and $\psi$ is an unramified character of $F$. Let
$(\pi,V_\pi)$ and $(\pi',V_{\pi'})$ be two $\psi$-generic
representations of $\SL_2(F)$ with the same central character.
\begin{enumerate}
\item If $\gamma(s,\pi,\eta,\psi)=\gamma(s,\pi',\eta,\psi)$ for all quasi-characters $\eta$ of $F^\times$,
then $\pi\cong \pi'$.
\item There is an integer $l=l(\pi,\pi')$ such that if $\eta$ is a quasi-character of $F^\times$ with
conductor $\cond(\eta)>l$, then
$$\gamma(s,\pi,\eta,\psi)=\gamma(s,\pi',\eta,\psi).$$
\end{enumerate}
\end{thm}
The proof of this result follows closely to \cite{Ba1,Ba2,Zh}, and
Howe vectors play an important role. With the help of this result,
combining with a nonvanishing result on archimedean local integrals
proved in Lemma 4.9, we follow the argument in Theorem 7.2.13 in
\cite{Ba2}, or Theorem 2 in \cite{Ca}, to prove the main global
result of this paper.
\begin{thm}[Theorem 4.8, Strong Multiplicity One for $\SL_2$]
Let $\pi=\otimes \pi_v$ and $\pi'=\otimes \pi'_v$ be two irreducible
cuspidal automorphic representation of $\SL_2(\BA)$ with the same
central character. Suppose that $\pi$ and $\pi'$ are both
$\psi$-generic. Let $S$ be a finite set of \textbf{finite} places
such that no place in $S$ is above $2$. If $\pi_v\cong \pi'_v$ for
all $v\notin S$, then $\pi =\pi'$.
\end{thm}
We remark here the restriction on residue characteristic comes from
Lemma 3.3. It is expected that this restriction can be removed.
\\

Besides the above, we also in this paper include a discussion of
relations between globally genericness and locally genericness. An irreducible cuspidal automorphic
representation $(\pi, V_\pi)$ is called globally generic if for some
$\phi\in V_\pi$, the integral
\[
\int_{N(F)\setminus N(\BA)} \phi(ug)\psi^{-1}(u)du\neq 0
\]
for some $g\in \SL_2(\BA)$. The representation $\pi$ is
called locally generic if each of its local component is generic. It is
easy to see that if $\pi$ is globally generic, then $\pi$ is also
locally generic. It is a conjecture that on reductive algebraic
group $G$, the converse is also true. This conjecture is closely
related to Ramanujan conjecture. See \cite{Sh} for more detailed
discussions. In this paper, we confirm this conjecture for $\SL_2$.
\begin{thm}[Theorem 4.3]
Let $\pi=\otimes_v\pi_v$ be an irreducible cuspidal automorphic
representation of $\SL_2(\BA)$ and $\psi=\otimes \psi_v$ be a
nontrivial additive character of $F\setminus \BA$. Then $\pi$ is
$\psi$-generic if and only if each $\pi_v$ is $\psi_v$-generic.
\end{thm}
In \cite{GeRS} Proposition 2.5, Gelbart, Rogawski and Soudry proved
similar results for $\RU(1,1)$ and for endoscopic cuspidal
automorphic representation of $\RU(2,1)$. From the discussions given
in \cite{GeRS}, Theorem 1.3 follows from the results of \cite{LL}
directly. Here, we insist to include this result because we adopt a
local argument (see Proposition 2.1) which is different from that
given in \cite{LL}. Hopefully, this local argument can be extended
to more general groups.

As explained above, there is essentially nothing new in this paper.
All the results and proofs should be known to the experts. Our task
here is simply to try to write down the details and to check
everything works out as expected.

This paper is organized as follows. In section 2, we collect basic
results about the local zeta integrals which will be needed. In
section 3, we study the Howe vectors and use it to prove the local
converse theorem and stability of local gamma factors. In section 4,
we prove the main global results.\\

\noindent \textbf{Acknowledgement:} We would like to thank our advisor Professor Jim Cogdell for his guidance, encouragement and support over the years. We would also like to thank Professor Tonghai Yang for helpful discussions. \\

\noindent \textbf{Notations:} Let $F$ be a field. In $\SL_2(F)$, we consider the following subgroups. Let $B$ be the upper triangular subgroup.
Let $B=TN$ be the Levi decomposition, where $T$ is the torus and $N$ is the upper triangular unipotent.
Denote $$t(a)=\begin{pmatrix}a& \\ & a^{-1} \end{pmatrix}\in T, \textrm{ for }a\in F^\times n(b)=\begin{pmatrix}1& b\\ &1 \end{pmatrix}\in N, \textrm{ for }b\in F.$$
Let $\bar N$ be the lower triangular unipotent and denote $$\bar n(x)=\begin{pmatrix}1&\\ x&1 \end{pmatrix}.$$
Let $w=\begin{pmatrix}&1\\ -1& \end{pmatrix}$.

Denote $St$ the natural inclusion $\SO_3(\BC)\ra \GL_3(\BC)$ and view it as the ``standard" representation of
${}^L\! \SL_2=\SO_3(\BC)$.

\section{The local zeta integral}
\subsection{The genericness of representations of $\SL_2(F)$}
In this section, let $F$ be a local field, and $\psi$ be a
nontrivial additive character of $F$, which is also viewed as a
character of $N(F)$.  For $\kappa\in F^\times$, $g\in \SL_2(F)$, we
define
$$g^\kappa=\begin{pmatrix}\kappa&\\ &1 \end{pmatrix}g \begin{pmatrix} \kappa^{-1}& \\ &1 \end{pmatrix}.$$
Explicitly, $$\begin{pmatrix}x& y\\ z&w \end{pmatrix}^\kappa=
\begin{pmatrix}x& \kappa y \\ \kappa^{-1} z & w \end{pmatrix}.$$
Note that if $\kappa\in F^{\times,2}$, say $\kappa=a^2$, then $g^\kappa=t(a)gt(a)^{-1}$,
i.e., $g\ra g^\kappa$ is an inner automorphism on $\SL_2(F)$.
Let $(\pi,V_\pi)$ be an infinite dimensional irreducible smooth representation of $\SL_2(F)$,
we consider the representation $(\pi^\kappa, V_{\pi^\kappa})$ defined by
$$V_{\pi^\kappa}=V_\pi, \textrm{ and } \pi^\kappa(g)=\pi(g^\kappa).$$

Let $\psi_\kappa$ be the character of $F$ defined by
$\psi_\kappa(b)=\psi(\kappa b)$. If $(\pi,V_\pi)$ is $\psi$-generic
with a nonzero $\psi$ Whittaker functional $\Lambda:V_\pi\to \BC$,
one verifies that
\[
\Lambda(\pi^\kappa(n)v)=\Lambda(\pi(n^\kappa)v)=\psi_\kappa(n)\Lambda(v)
\]
for all $n\in N(F), v\in V_{\pi^\kappa}=V_\pi$. Hence
$(\pi^\kappa,V_{\pi^\kappa})$ is $\psi_\kappa$-generic. \\

\begin{prop}
If $\pi$ is both $\psi$- and $\psi_\kappa$-generic, then $\pi\cong \pi^\kappa$.
\end{prop}
\begin{proof}
If $F$ is non-archimedean, we proved the same result in the
$\RU(1,1)$-case in \cite{Zh}, Theorem 2.12. The proof in the
$\SL_2$-case is the same.

If $F$ is archimedean, the case $F=\BC$ is easy, as every $\kappa$
has a square root in $\BC$. Now consider $F=\BR$, and we will work
with the category of smooth representations of moderate growth of
finite length. The Whittaker functional is an exact functor from
this category to the category of vector spaces by Theorem 8.2 in
\cite{CMH}.

We first consider the case when $\pi=\Ind_B^G(\chi)$ for some
quasi-character $\chi$ of $F^\times$. For $f\in I(\chi)$, consider
the function $f^\kappa$ on $\SL_2(F)$ defined by
$f^\kappa(g)=f(g^{\kappa^{-1}})$. It is clear that $f^\kappa\in
I(\chi)^\kappa$ and the map $f\mapsto f^\kappa$ defines an
isomorphism $I(\chi)\ra I(\chi)^\kappa$.

By results in Chapter 2 of \cite{Vo}, if $\pi$ is not a fully
induced representation, then it can be embedded into a principal
series $I(\chi)$. This $I(\chi)$ has two irreducible infinite
dimensional subrepresentations, and use $\pi'$ to denote the other
one. The quotient of $I(\chi)$ by the sum of $\pi$ and $\pi'$,
denoted by $\pi''$, is finite dimensional, i.e., we have an exact
short sequence
\[
0\to \pi\oplus \pi'\to I(\chi)\to \pi''\to 0
\]

First by Theorem 6.1 in \cite{CMH}, we know that the dimension of
Whittaker functionals on $I(\chi)$ is one dimensional for either
$\psi$ or $\psi_\kappa$. Note that $\pi''$ cannot be generic as it
is finite dimensional. Since the Whittaker functor is exact, it
follows that the dimensional of Whittaker functionals on $\pi\oplus
\pi'$ is also one dimensional for either $\psi$ or $\psi_\kappa$. By
the assumption $\pi$ is both $\psi$ and $\psi_\kappa$ generic, thus
$\pi'$ is neither $\psi$ nor $\psi_\kappa$ generic.

Now since $\pi$ is $\psi$ generic, then $\pi^\kappa$ is
$\psi_\kappa$ generic. Hence the image of $\pi$ under the
isomorphism $f\to f^\kappa$ between $I(\chi)\ra I(\chi)^\kappa$ is
again $\psi_\kappa$ generic, hence it has to be $\psi$ generic and
isomorphic to $\pi$, which finishes the proof.
\end{proof}

%Let $F$ be a local field, and $\psi$ be a non-trivial additive character of $F$.

\subsection{Weil representations of $\widetilde \SL_2$}
Let $\widetilde \SL_2$ be the metaplectic double cover of $\SL_2$. Then we have an exact sequence
$$0\ra \mu_2\ra \widetilde \SL_2 \ra \SL_2\ra 0,$$
where $\mu_2=\wpair{\pm 1}$.

The product on $\widetilde \SL_2(F)$ is given by
$$(g_1,\zeta_1)(g_2,\zeta_2)=(g_1 g_2, \zeta_1 \zeta_2 c(g_1,g_2)),$$
where $c:\SL_2(F)\times \SL_2(F)\ra \wpair{\pm1}$ is defined by
$$c(g_1,g_2)=(\bx(g_1),\bx(g_2))_F (-\bx(g_1)\bx(g_2), \bx(g_1g_2))_F,$$
where $$\bx\begin{pmatrix} a& b\\ c& d \end{pmatrix}=\left\{\begin{array}{lll} c, & c\ne 0, \\ d, & c=0. \end{array}
\right.$$
For these formulas, see \cite{Sz} for example.

For a subgroup $A$ of $\SL_2(F)$, we denote $\tilde A$ the preimage
of $A$ in $\widetilde \SL_2(F)$, which is a subgroup of $\widetilde
\SL_2(F)$. It is easy to see that $\widetilde N(F)=N(F)\times \mu_2$
and $\widetilde{\bar N}(F)=\bar N(F)\times \mu_2$. For an element
$g\in \SL_2(F)$, we sometimes write $(g,1)\in \widetilde \SL_2(F)$
as $g$ by abuse of notation.

A representation $\pi$ of $\widetilde \SL_2(F)$ is called genuine if
$\pi(\zeta g)=\zeta \pi(g)$ for all $g\in \widetilde \SL_2(F)$ and
$\zeta \in \mu_2$. Let $\psi$ be an additive character of $F$, there
is Weil representation $\omega_{\psi}$ of $\widetilde \SL_2(F)$ on
$\CS(F)$, the Bruhat-Schwartz functions on $F$. For $f\in \CS(F)$,
we have the familiar formulas:
\begin{align*}
\left(\omega_{\psi}\begin{pmatrix}&1\\ -1& \end{pmatrix}\right)f(x)&=\gamma(\psi)\hat f(x),\\
\left(\omega_{\psi}\begin{pmatrix}1& b\\ &1 \end{pmatrix}\right) f(x)&=\psi(bx^2)f(x), b\in F\\
\left(\omega_{\psi}\begin{pmatrix}a& \\ &a^{-1} \end{pmatrix}\right) f(x)&=|a|^{1/2}
\frac{\gamma(\psi)}{\gamma(\psi_a)}f(ax), a\in F^\times,
\end{align*}
and
$$ \omega_{\psi}(\zeta)f(x)=\zeta f(x), \zeta \in \mu_2. $$
Here $\hat f(x)=\int_F f(y)\psi(2xy)dy$, where $dy$ is normalized so
that $(\hat f)^{\hat~}(x)=f(-x)$, $\gamma(\psi)$ is the Weil index
and $\psi_a(x)=\psi(ax)$.

Let $\widetilde T$ be the inverse image of
$T=\wpair{t(a):=\begin{pmatrix} a& \\ & a^{-1}\end{pmatrix}, a\in
F^\times }\subset \SL_2(F)$ in $\widetilde \SL_2(F)$. The product in
$\widetilde T$ is given by the Hilbert symbol, i.e.,
$$(t(a),\zeta_1)(t(b),\zeta_2)=(t(ab), \zeta_1\zeta_2(a,b)_F),$$
where $(a,b)_F$ is the Hilbert symbol.
 The function
$$\mu_\psi(a)=\frac{\gamma(\psi)}{\gamma(\psi_a)}$$
satisfies
$$\mu_{\psi}(a)\mu_\psi(b)=\mu_\psi(ab)(a,b),$$
and thus defines a genuine character of $\widetilde T$.

The representation $\omega_\psi$ is not irreducible, and we have
$\omega_{\psi}=\omega_\psi^+ \oplus \omega_{\psi}^-$, where
$\omega_\psi^+$ (resp. $\omega_\psi^-$) is the subrepresentation on
even (resp. odd) functions in $\CS(F)$. All the above facts can be
found in section 1 in \cite{GPS1} for example.

\subsection{The local zeta integral}

 Let $\mu_\psi(a)=\frac{\gamma(\psi)}{\gamma(\psi_a)}$, which is viewed as a character of $\widetilde T$.
 Let $\eta$ be a character of $F^\times$, and we consider the induced representation
 $I(s,\eta,\psi)=\Ind_{\widetilde B(F)}^{\widetilde \SL_2(F)}(\eta_{s-1/2} \mu_\psi)$.

Let $(\pi,V)$ be a $\psi$-generic representation of $\SL_2(F)$ with
its Whittaker model $\CW(\pi,\psi)$. Choose $W\in \CW(\pi,\psi)$
$\phi\in \CS(F)$ and $f_s\in I(s,\eta,\psi^{-1})$, note that
$(\omega_{\psi^{-1}}(h)\phi)(1)f_s(h)$ is well-defined as a function
on $\SL_2(F)$, and we consider the integral
$$\Psi(W, \phi, f_s)=\int_{N(F)\setminus \SL_2(F)}W(h)(\omega_{\psi^{-1}}(h)\phi)(1)f_s(h)dh.$$
By results in section 5 and section 12 in \cite{GPS2}, the above
integral is absolutely convergent when $Re(s)$ is large enough and
has a meromorphic continuation to the whole plane.

\noindent\textbf{Remark:} In \cite{GPS2}, Gelbart and Piatetski-Shapiro constructed a global
zeta integral for $\Sp_{2n}\times \GL_n$, showed that it is Eulerian and sketched a proof of the
local functional equation. This is the so-called method C in \cite{GPS2}. The above integral
is the easiest case of the Gelbart and Piatetski-Shapiro integral, namely, when $n=1$.

\subsection{local functional equation}
The trilinear form $(W,\phi,f_s)\mapsto \Psi(W,\phi,f_s)$ defines an element in
$$\Hom_{\SL_2}(\pi\times \omega_{\psi^{-1}}\otimes I(s,\eta,\psi^{-1}),\BC),$$
which has dimension at most one. The proof of this fact is given in
\cite{GPS2}(section 11) and also can be deduced by the uniqueness of
Fourier-Jacobi model for $\SL_2$, see \cite{Su}. Let
$$M_s: I(s,\eta,\psi^{-1})\ra I(1-s,\eta^{-1},\psi^{-1})$$
be the standard intertwining operator, i.e., $$M_s(f_s)(g)=\int_N f_s(w ng)dn.$$

By the uniqueness of the above Hom space, we then get the following
\begin{prop}
There is a meromorphic function $\gamma(s,\pi,\eta,\psi)$ such that
$$\Psi(W,\phi, M_s(f_s))=\gamma(s,\pi,\eta,\psi)\Psi(W,\phi,f_s),$$
for all $W\in \CW(\pi,\psi)$, $\phi\in \CS(F)$ and $f_s\in I(s,\eta,\psi^{-1})$.
\end{prop}

\subsection{Unramified calculation}
The unramified calculation of Method C is in fact not included in \cite{GPS2}, but it is really
easy to do that in our super easy case.

Suppose everything is unramified. Suppose that the residue
characteristic is not 2 so that $\mu_\psi$ is unramified, see
\cite{Sz}. Suppose the representation $(\pi,V)$ has Satake parameter
$a$, which means that $\pi$ is the unramified component
$\Ind_{B(F)}^{\SL_2(F)}(\nu)$ for an unramified character $\nu$ and
$a=\nu(p_F)$, where $p_F$ is some prime element of $F$. Let
$$b_k=\diag(p_F^k, p_F^{-k}),$$
and $W$ be the spherical Whittaker functional normalized by $W(e)=1$. Then $W(b_k)=0$ for $k<0$, and
$$W(b_k)=\frac{q^{-k}}{a-1} (a^{k+1}-a^{-k}),$$
by the general Casselman-Shalika formula.
For $k\ge 0$, we have
$$(\omega_{\psi^{-1}}(b_k)\phi)(1)=\mu_{\psi^{-1}}(p_F^k)|p_F^k|^{1/2},$$
where $\phi$ is the characteristic function of $\CO_F$. On the other hand, let $f_s$ be the standard
spherical section of $I(s,\eta,\psi^{-1})$ normalized by $f_s(1)=1$. Then we have
$$f_s(b_k)=\eta(p_F^k)|p_F^k|^{s+1/2}\mu_{\psi^{-1}}(p_F^k).$$
Since $\mu_{\psi^{-1}}(p_F^k)\mu_{\psi^{-1}}(p_F^k)=(p_F^k, p_F^k)_F=(p_F,-1)_F^k,$
We have
\begin{align*}\Psi(W,\phi,f_s)=& \int_{F^\times}\int_K W(t(a)k)\omega_{\psi^{-1}}(t(a k)\phi)(1)f_s(t(a)k) |a|^{-2}dk da\\
&=\int_{F^\times} W(t(a))\omega_{\psi^{-1}}(t(a))\phi(1)f_s(t(a))|a|^{-2}da\\
&= \sum_{k\ge 0} W(b_k)(\omega_{\psi^{-1}}(b_k)\phi)(1)f_s(b_k) |p_F^k|^{-2}\\
&=\frac{1}{a-1}\sum_{k\ge 0} (a^{k+1}-a^{-k})(p_F,-1)^k\eta(p_F)^k q_F^{-ks}\\
&=\frac{1+c}{(1-ac)(1-a^{-1}c)}=\frac{1-c^2}{(1-c)(1-ac)(1-a^{-1}c)}\\
&=\frac{L(s,\pi,St\otimes \eta \chi)}{L(2s,\eta^2)},
\end{align*}
where $c=(p_F,-1)\eta(p)q_F^{-s},$ and $\chi(a)=(a,-1)_F$. Recall that $St$ is the standard representation of ${}^L\SL_2=\SO_3(\BC)$.

\noindent \textbf{Remark:} Let $M_s: I(s,\eta,\psi^{-1})\ra I(1-s,\eta^{-1},\psi^{-1})$ be the
standard intertwining operator, i.e., $$M_s(f_s)(g)=\int_N f_s(w ng)dn.$$
From the calculation given in \cite{Sz}, it is easy to see that
$$M_s(f_s)=\frac{L(2s-1,\eta^2)}{L(2s,\eta^2)}f_{1-s},$$
where $f_s$ (resp. $f_{1-s}$) is the standard spherical section in $I(s,\eta,\psi)$
(resp. $I(1-s,\eta^{-1},\psi)$). Thus the factor $L(2s,\eta^2)$ appeared in the above
unramified calculation will play the role of the normalizing factor of a global intertwining
operator or Eisenstein series.

\section{Howe vectors and the local converse theorem}

In this section, we assume the residue characteristic of $F$ is not
2. We will follow Baruch's method, \cite{Ba1} and \cite{Ba2}, to
give a proof of local converse theorem for generic representations
of $\SL_2(F)$.

\subsection{Howe vectors} Let $\psi$ be an unramified character.  For a positive integer $m$, let
$K_m=(1+M_{2\times 2}(\CP_F^m))\cap \SL_2(F)$ where $\CP_F$ denotes
the maximal ideal in $\CO_F$. Define a character $\tau_m$ of $K_m$
by
$$\tau_m(k)=\psi(p_F^{-2m}k_{12}).$$
for $k=(k_{ij})\in K_m$. It is easy to see that $\tau_m$ is indeed a
character on $K_m$.

Let $d_m=\diag(p_F^{-m},p_F^{m}).$ Consider the subgroup $J_m=d_m
K_m d_m^{-1}$. Then
$$J_m=\begin{pmatrix}1+\CP_F^m & \CP_F^{-m}\\ \CP^{3m}& 1+\CP_F^m \end{pmatrix}.$$

Define $\psi_m(j)=\tau_m(d_m^{-1}j d_m)$ for $j\in J_m$. For a
subgroup $H\subset \SL_2(F)$, denote $H_m=H\cap J_m$. It is easy to
check that $\psi_m|_{N_m}=\psi|_{N_m}$.

Let $\pi$ be an irreducible smooth $\psi$-generic representation of
$\SL_2(F)$, let $v\in V_\pi$ be a vector such that $W_v(1)=1$. For
$m\ge 1$, as in \cite{Ba1} and \cite{Ba2}, we consider
\begin{equation}{\label{eq32}}v_m=\frac{1}{\Vol(N_{m})}\int_{N_{m}}\psi(u)^{-1}\pi(n)vdn.\end{equation}
Let $L\ge 1$ be an integer such that $v$ is fixed by $K_L$.

\begin{lem}{\label {lem31}} We have
\begin{enumerate}
\item $W_{v_m}(1)=1.$
\item If $m\ge L$, $\pi(j)v_m=\psi_m(j)v_m$ for all $j\in J_m$.
\item If $k\le m$, then
$$v_m=\frac{1}{\Vol(N_{m})}\int_{N_{m}}\psi(u)^{-1}\pi(u)v_kdu.$$
\end{enumerate}
\end{lem}
The proof of this lemma is the same as the proof in the $\RU(2,1)$ case, which is given in \cite{Ba2}.
\begin{lem}
Let $m\ge L$ and $t=t(a)$ for $a \in F^\times$. Then
\begin{enumerate}
\item if $W_{v_m}(t)\ne 0$, we have
$$a^2\in 1+\CP_F^m;$$
\item if $W_{v_m}(tw)\ne 0$, we have
$$a^2\in \CP^{-3m}.$$
\end{enumerate}
\end{lem}
\begin{proof}
(1) Take $x\in \CP^{-m}$, we then have $ n(x)\in N_m\subset J_m$. From the relation
$$ tn(x)=n(a^2x) t,$$
and the above lemma, we have
$$\psi(x)W_{v_m}(t)=\psi(a^2x)W_{v_m}(t).$$
If $W_{v_m}(t)\ne 0$, we get $\psi(x)=\psi(a^2)x$ for all $x\in \CP^{-m}$. Since $\psi$ is unramified,
we get $a^2\in 1+\CP^m$.

(2) For $x\in \CP^{3m},$ we have $\bar n(x)\in \bar N_m$. From the relation $tw\bar n(x)=n(-a^2x) tw$
and the above lemma, we get
$$W_{v_m}(tw)=\psi(-a^2x)W_{v_m}(tw).$$
Thus if $W_{v_m}(tw)\ne 0$, we get
$\psi(-a^2x)=1$ for all $x\in \CP^{3m}$. Thus $a^2\in \CP^{-3m}$.
\end{proof}

\begin{lem}
For $m\ge 1$, consider the square map ${}^ 2: 1+\CP^m\ra 1+\CP^m$, $a\mapsto a^2$ is well-defined and surjective.
\end{lem}
This lemma requires the residue field of $F$ is not of characteristic 2 which we assumed throughout this section.
\begin{proof}
For $x\in \CP^m$, it is clear that $(1+x)^2=1+2x+x^2\in 1+\CP^m$.
Thus the square map is well-defined. On the other hand, we take
$u\in 1+\CP^m$ and consider the equation $f(X):=X^2-u=0$. We have
$f'(X)=2X$. Since $q^{-m}=|1-u|=|f(1)|<|f'(1)|^2=|2|^2=1$, by
Newton's Lemma, Proposition 2, Chapter II of \cite{Lg}, there is root $a\in \CO_F$ of $f(X)$ such that
$$|a-1|\le \frac{|f(1)|}{|f'(1)|^2}=|1-u|=q^{-m}.$$
Thus we get a root $a\in 1+\CP^m$ of $f(X)$. This completes the proof.
\end{proof}

Let $Z=\wpair{\pm 1}$, and identify $Z$ with the center of
$\SL_2(F)$. Use $\omega_\pi$ to denote the central character of
$\pi$.
\begin{cor}
Let $m\ge L$, then we have
$$W_{v_m}(t(a))=\left\{\begin{array}{lll} \omega_\pi(z), & \textrm{ if } a=za',
\textrm{ for some }z\in Z, a'\in 1+\CP^m,\\ 0, & \textrm{ otherwise. } \end{array}\right.$$
\end{cor}
\begin{proof}
Suppose that $W_{v_m}(t(a))\ne 0$, then by Lemma 3.2, we have
$a^2\in 1+\CP^m$. By Lemma 3.3, there exists an $a'\in 1+\CP^m$ such
that $a^2=(a')^2$. Thus $a=za'$ for some $z\in Z$. Since $a'\in
1+\CP^m$, we get $t(a')\in J_m$. The assertion follows from Lemma
3.1.
\end{proof}

From now on, we fix two $\psi$-generic representations $(\pi,
V_{\pi})$ and $(\pi', V_{\pi'})$ with the same central characters.
Fix $v,v'$ such that $W_{v}(1)=1=W_{v'}(1)$. Let $L$ be an integer
such that both $v$ and $v'$ are fixed by $K_L$. For $m\ge 1$, we
consider the Howe vectors $v_m$ and $v'_m$.

By Corollary 3.4 and the fact that $\omega_{\pi}=\omega_{\pi'}$, we
get
\begin{cor}{\label{cor35}}
For $m\ge L$, we have $W_{v_m}(g)=W_{v_m'}(g)$ for all $g\in B$.
\end{cor}

\begin{lem}[Baruch]
If $m\ge 4L$ and $n\in N-N_m$, we have
$$W_{v_m}(tw n)=W_{v_m'}(twn),$$
for all $t\in T$.
\end{lem}
\begin{proof}
It's completely the same as Proposition 3.4 in \cite{Zh}, which is also a special case of Lemma 6.2.2 of \cite{Ba1}. Similar type result for $\RU(2,1)$ can be found in \cite{Ba2}. We just remark that the proof of this lemma depends on Corollary \ref{cor35}, and hence require the residue characteristic of $F$ is not 2.
\end{proof}
\subsection{Induced representations}
Note that $\bar N(F)$ and $N(F)$ splits in $\widetilde \SL_2(F)$.
Moreover, for $g_1\in N$ and $g\in \bar N$ we have $c(g_1,g_2)=1$.
In fact, if $g_1=n(y)$ and $g_2=\bar n(x)$ with $x\ne 0$, we have
$\bx(g_1)=1$ and $\bx(g_2)=x$, and thus
$$c(g_1,g_2)=( 1,x)_F(-x, x)_F=1.$$
This shows that $N(F) \cdot \bar N(F)\subset \SL_2(F),$ where
$\SL_2(F)$ denotes the subset of $\widetilde \SL_2(F)$ which
consists elements of the form $(g,1)$ for $g\in \SL_2(F)$.

Let $X$ be an open compact subgroup of $N(F)$. For $x\in X$ and
$i>0$, we consider the set $A(x,i)=\wpair{\bar n\in \bar N(F): \bar
n x\in B \cdot \bar N_{i}}$.
\begin{lem}
\begin{enumerate}
\item For any positive integer $c$, there exists an integer $i_1=i_1(X,c)$ such that for all
$i\ge i_1$, $x\in X$ and $ \bar n\in A(x,i)$, we have
$$\bar n x=n t(a) \bar n_0$$
with $n\in N, \bar n_0 \in \bar N_{i}$ and $a\in 1+\CP^c$.
\item There exists an integer $i_0=i_0(X)$ such that for all $i\ge i_0$, we have $A(x,i)=\bar N_{i}$
for all $i\ge i_1$.
\end{enumerate}
\end{lem}
\begin{proof}
By abuse of notation, for $x\in X$, we write $x=n(x)$. Since $X$ is compact, there is a constant $C$
such that $|x|<C$ for all $n(x)\in X\subset N$.

For $n(x)\in X, \bar n(y)\in A(x,i)$, we have $\bar n(y)n(x) \in B \cdot \bar N_{i}$, thus we can
assume that $$\bar n(y)n(x)=\begin{pmatrix}a& b\\ & a^{-1} \end{pmatrix} \bar n(\bar y)$$ for
$a\in F^\times, b\in F$ and $\bar y \in \CP^{3i}$. Rewrite the above expression as
$$\bar n (-y) \begin{pmatrix} a& b \\ & a^{-1} \end{pmatrix}= n(x)\bar n(-\bar y),$$
or

$$\begin{pmatrix}a & b\\ -ay & a^{-1}-by \end{pmatrix}=\begin{pmatrix}1-x\bar y & x\\ -\bar y &1 \end{pmatrix}.$$
Thus we get $$a=1-x\bar y, ay=\bar y.$$
Since $|x|<C$ and $\bar y\in \CP^{3i}$, it is clear that for any positive integer $c$, we can
choose $i_1(X,c)$ such that $a=1-x\bar y\in 1+\CP^c$ for all $ n(x)\in X$ and $\bar n(y)\in A(x,i)$. This proves (1).

If we take $i_0(X)=i_1(X,1)$, we get $a\in 1+\CP\subset \CO^\times$ for $i\ge i_0$. From $ay=\bar y$,
we get $y\in \CP^{3i}$. Thus we get that for $i\ge i_0(X)$, we have $ \bar n(y)\in \bar N_{i}$, i.e.,
$A(x,i)\subset \bar N_{i}$.

The other direction can be checked similarly if $i$ is large. We omit the details.
\end{proof}

Given a positive integer $i$ and a complex number $s\in \BC$, we
consider the following function $f^i_s$ on $\widetilde \SL_2(F)$:
$$f_s^i((g,\zeta))=\left\{\begin{array}{lll}\zeta \mu_{\psi^{-1}}(a)\eta_{s+1/2}(a),& \textrm{ if } g=
\left(\begin{pmatrix}a & b\\ & a^{-1}\end{pmatrix}, \zeta \right)  \bar n(x), \textrm{ with }
a\in F^\times, b\in F, \zeta\in \mu_2, x\in \CP^{3i}, \\ 0, &\textrm{ otherwise.} \end{array}\right.$$

\begin{lem}
\begin{enumerate}
\item There exists an integer $i_2(\eta)$ such that for all $i\ge i_2(\eta)$, $f_s^i$ defines a section
in $I(s,\eta,\psi^{-1})$.
\item Let $X$ be an open compact subset of $N$, then there exists an integer $I(X,\eta)\ge i_2(\eta)$
such that for all $i\ge I(X,\eta)$, we have
$$\tilde f_s^i(w x)=\vol(\bar N_{i})=q_F^{-3i},$$
for all $x\in X$, where $\tilde f_s^i=M_s(f_s^i)$.
\end{enumerate}
\end{lem}
Here $w=\begin{pmatrix}&1\\ -1& \end{pmatrix}$.
\begin{proof}
(1) From the definition, it is clear that $$f_s^i\left(
\left(\begin{pmatrix} a&b\\ & a^{-1}\end{pmatrix}, \zeta\right)
\tilde g \right)=\zeta\mu_{\psi^{-1}}(a)\eta_{s+1/2}(a)f_s^i(\tilde
g),$$ for $a\in F^\times, b\in F, \zeta \in \mu_2, $ and $\tilde
g\in \widetilde \SL_2(F)$. It suffices to show that for $i$ large,
there is an open compact subgroup $\widetilde H_i\subset \widetilde
\SL_2(F)$ such that $f_s^i(\tilde g \tilde h)=f_s^i (\tilde g)$ for
all $\tilde g\in \widetilde \SL_2(F),$ and $\tilde h\in \widetilde
H_i$.

If $\psi$ is unramified and the residue characteristic is not 2 as we assumed, the character
$\mu_{\psi^{-1}}$ is trivial on $\CO_F^\times$, see \cite{Sz} for example.

Let $c$ be a positive integer such that $ \eta$ is trivial on $1+\CP^c$. Let $i_2(\eta)=
\wpair{c, i_0(N\cap K_c ), i_1(N\cap K_c, c) }.$ For $i\ge i_2(\eta)$, we take
$\widetilde H_i= K_{4i}=1+M_2(\CP^{4i})$. Note that $K_{4i}$ splits, and thus can be viewed
as a subgroup of $\widetilde \SL_2$. We now check that for $i\ge i_2(\eta)$, we have
$f_s^i(\tilde g h)=f_s(\tilde g)$ for all $\tilde g\in \widetilde \SL_2$ and $h\in K_{4i}$.
We have the Iwahori decomposition $K_{4i}=(N\cap K_{4i})(T\cap K_{4i})(\bar N\cap K_{4i})$.
For $h\in \bar N\cap K_{4i}\subset \bar N_i$, it is clear that $f_s^i(\tilde g h)=f_s^i(\tilde g)$
by the definition of $f_s^i$. Now we take $h\in T\cap K_{4i}$. Write $h=t(a_0)$, with $a_0\in 1+\CP^{4i}$.
We have $\bar n(x) h= h \bar n(a_0^{-2} x)$. It is clear that $x\in \CP^{3i}$ if and only if $a_0^{-2}x\in \CP^{3i}$.
On the other hand, for any $a\in F^\times, b\in F$, we have
$$c\left( \begin{pmatrix} a& b\\ & a^{-1} \end{pmatrix}, t(a_0) \right)=(a,a_0)=1,$$
since $a_0\in 1+\CP_F^{4i}\subset F^{\times, 2}$, by Lemma 3.3. Thus
we get
$$ \left( \begin{pmatrix} a& b\\ & a^{-1} \end{pmatrix}, \zeta \right)\bar n(x)h=
\left( \begin{pmatrix} aa_0& ba_0^{-1}\\ & a^{-1}a_0^{-1} \end{pmatrix}, \zeta \right)\bar n(a_0^{-2}x). $$
By the definition of $f_s^i$, if $x\in \CP^{3i}$, for $g=\left(\begin{pmatrix} a& b\\ & a^{-1} \end{pmatrix},
\zeta \right)\bar n(x)$ we get
$$f_s^i(g h)=\mu_{\psi^{-1}}(a_0a)\eta_{s+1/2}(aa_0) =\mu_{\psi^{-1}}(a) \eta_{s+1/2}(a)=f_s^i(g),$$
by the assumption on $i$.

Finally, we consider $h\in  N\cap K_{4i}\subset N\cap K_c$. By assumption on $i$, we get
$$A(h,i)=A(h^{-1},i)=\bar N_i.$$
In particular, for $\bar n\in \bar N_i$, we have $ \bar n h \in B\cdot \bar N_i$ and
$\bar n h^{-1}\in B\cdot \bar N_i$. Now it is clear that $\tilde g\in \widetilde B\cdot \bar N_i$
if and only if $\tilde g h\in \widetilde B\cdot \bar N_i$. Thus $f_s^i(\tilde g)=0$ if and only if
$f_s^i(\tilde gh)=0$. Moreover, for $\bar n \in \bar N_i$, we have
$$\bar n h=\begin{pmatrix}a_0 & b_0\\ & a_0^{-1} \end{pmatrix}\bar n_0,$$
for $a_0\in 1+\CP^c$, $b_0\in F$ and $\bar n_0\in \bar N_i$. Thus for $\tilde g=
\left(\begin{pmatrix}a&b \\ & a^{-1}  \end{pmatrix}, \zeta\right) \bar n$ with $\bar n \in \bar N_i$,
we get
$$\tilde gh =\left( \begin{pmatrix}aa_0 & ab_0+a_0^{-1}b \\ & a_0^{-1}a^{-1} \end{pmatrix},\zeta \right)
\bar n_0.$$
Here we used the fact that $a_0\in 1+\CP^c$ is a square, and thus
$$c\left(\begin{pmatrix}a& b\\ & a^{-1} \end{pmatrix},\begin{pmatrix}a_0 & b_0\\ & a_0^{-1} \end{pmatrix}  \right)=1.$$

Since $\mu_{\psi^{-1}}(a_0)=1$, $(a,a_0)=1$ and $\eta_{s+1/2}(a_0)=1$, we get
$$f_s^i(\tilde g h)=f_s^i(g).$$
This finishes the proof of (1).

(2)   As in the proof of (1), let $c$ be a positive integer such that $\eta$ is trivial on $1+\CP^c$.
Take $I(X,\eta)=\max\wpair{i_1(X,c), i_0(X)}$. We have $$\tilde f_s^i(wx)=\int_{N} f_s^i(w^{-1} n w x)dn.$$
By the definition of $f_s^i$, $f_s^i(w^{-1} n w x)\ne 0$ if and only if $w^{-1} n w x\in B\bar N_{i}$,
if and only if $w^{-1} n w \in A(x,i)=\bar N_{i}$ for all $i\ge I(X),$ and $ x\in X$.
On the other hand, if $w^{-1} n w\in A(x,i)$, we have
$$w^{-1} n wx= \begin{pmatrix} a & b\\ & a^{-1}\end{pmatrix}\bar n_0, $$
with $a\in 1+\CP^c$. Thus
$$f_s^i( w^{-1}n wx)=\eta_{s+1/2}(a)\mu_{\psi^{-1}}(a)=1.$$ Now the assertion is clear.
\end{proof}

\subsection{The local converse theorem}
\begin{lem}
Let $\phi^m$ be the characteristic function of $1+\CP^m$. Then
\begin{enumerate}
\item for $n\in N_m$, we have $\omega_{\psi^{-1}}(n)\phi^m=\psi^{-1}(n)\phi^m;$
\item for $\bar n\in \bar N_m$, we have $\omega_{\psi^{-1}}(\bar n)\phi^m=\phi^m$.
\end{enumerate}
\end{lem}
\begin{proof}
(1) For $n=n(b)\in N_m$, we have $b\in \CP^{-m}$. For $x\in 1+\CP^m$, we have $bx^2-b\in \CO_F$ and thus
$$\omega_{\psi^{-1}}(n)\phi^m(x)=\psi^{-1}(bx^2)\phi^m(x)=\psi^{-1}(b)\phi^m(x).$$
For $x\notin 1+\CP^m$, we have $\omega_{\psi^{-1}}(n)\phi^m(x)=\psi^{-1}(bx^2)\phi^m(x)=0$.
The first assertion follows.

(2) For $\bar n\in \bar N_m$, we can write $\bar n =w^{-1}n(b) w$ with $b\in \CP^{3m}$.
Let $\phi'=\omega_{\psi^{-1}}(w)\phi^m$. We have
\begin{align*}
\phi'(x)&=\gamma(\psi^{-1})\int_F \phi^m(y)\psi^{-1}(2xy)dy\\
&=\gamma(\psi^{-1})\psi^{-1}(2x)\int_{\CP^m}\psi^{-1}(2xz)dz\\
&=\gamma(\psi^{-1})\psi^{-1}(2x)\vol(\CP^m)
\textrm{Char}(\CP^{-m})(x)
\end{align*}
where $\textrm{Char}(\CP^{-m})$ denotes the characteristic function
of the set $\CP^{-m}$. It is clear that
$\omega_{\psi^{-1}}(n(b))\phi'=\phi'$. Thus we have
$$\omega_{\psi^{-1}}(\bar n)\phi^m=\omega_{\psi^{-1}}(w^{-1}n(b))\phi'=\omega_{\psi^{-1}}(w^{-1})\phi'=
\omega_{\psi^{-1}}(w^{-1})\omega_{\psi^{-1}}(w)\phi^m=\phi^m.$$
This completes the proof.
\end{proof}

\begin{thm}
Suppose that the residue characteristic of $F$ is not $2$ and $\psi$ is an unramified character of $F$.
Let $(\pi,V_\pi)$ and $(\pi',V_{\pi'})$ be two $\psi$-generic representations of $\SL_2(F)$ with the
same central character.
\begin{enumerate}
\item If $\gamma(s,\pi,\eta,\psi)=\gamma(s,\pi',\eta,\psi)$ for all quasi-characters $\eta$ of $F^\times$,
then $\pi\cong \pi'$.
\item There is an integer $l=l(\pi,\pi')$ such that if $\eta$ is quasi-character of $F^\times$ with
conductor $\cond(\eta)>l$, then
$$\gamma(s,\pi,\eta,\psi)=\gamma(s,\pi',\eta,\psi).$$
\end{enumerate}
\end{thm}
\noindent\textbf{Remark:} D. Jiang conjectured the local converse theorem for any reductive group $G$,
in \cite{Jng}, Conjecture 3.7. Our theorem can be viewed one example of that general conjecture.
\begin{proof}
We fix the notations $v\in V_\pi, v'\in V_{\pi'}$ and $L$ as before.

Let $\eta$ be a quasi-character of $F^\times$.  We take an integer $m\ge \wpair{6L, \cond(\eta)}$.
We consider the Howe vectors $v_m$ and $v'_m$. We also take an integer $i\ge \wpair{i_2(\eta), I(N_m,\eta), m}$.
In particular we have a section $f_s^i\in I(s,\eta,\psi)$. Let $W_m=W_{v_m}$ or $W_{v_m'}$.
 We compute the integral $\Psi(W_m, \phi^m,f_s^i)$ on the open dense subset $T\bar N(F)=N(F)\setminus N(F)T\bar N(F)$
 of $N(F)\setminus \SL_2(F)$. For $ g=nt(a)\bar n$, we can take the quotient measure as $dg=|a|^{-2} d\bar n da.$
 By the definition of $f_s^i$, we get
\begin{align*}
\Psi(W_m, \phi^m,f_s^i)&=\int_{T\times \bar N(F)}W_m(t(a)\bar
n)(\omega_{\psi^{-1}}(t(a)\bar n) \phi^m)
(1)f_s^i(t(a)\bar n) |a|^{-2}d\bar n da\\
&=\int_{T\times \bar N_i} W_m(t(a)\bar n)\mu_{\psi^{-1}}(a)|a|^{1/2}
\omega_{\psi^{-1}}(\bar n)\phi^m(a) \mu_{\psi^{-1}}(a) \eta_{s+1/2}(a) |a|^{-2}d\bar n da\\
&=\int_{T\times \bar N_i}W_m(t(a)\bar n)\omega_{\psi^{-1}}(\bar
n)\phi^m(a) \chi(a) \eta_{s-1}(a)d\bar n da,
\end{align*}
where $\chi(a)=\mu_{\psi^{-1}}(a)\mu_{\psi^{-1}}(a)=(a,-1)_F$. Since
$i\ge m$, we get $\bar N_i\subset \bar N_m$. By Lemma 3.1, and Lemma
3.9, we get $W_m(t(a)\bar n)=W_m(t(a))$ and $\omega_{\psi^{-1}}(\bar
n)\phi^m=\phi^m$. Thus we get
$$ \Psi(W_m, \phi^m,f_s^i)=q^{-3i}\int_{F^\times} W_m(t(a))\phi^m(a)\chi(a)\eta_{s-1}(a)da.$$
Since $\phi^m=\textrm{Char}(1+\CP^m)$, and for $a\in 1+\CP^m$, we
have $W_m(t(a))=1$ by Lemma 3.1, we get
$$\Psi(W_m,\phi^m, f_s^i)=q^{-3i}\int_{1+\CP^m}\chi(a)\eta(a)da.$$
Since $\chi(a)=1$ for $a\in 1+\CP^m$, and $ m\ge \cond(\eta)$ by assumption, we get
$$\Psi(W_m, \phi^m, f_s^i)=q^{-3i-m}.$$
The above calculation works for both $W_{v_m}$ and $W_{v_m'}$, thus we have
\begin{equation}\label{eq22}\Psi(W_{v_m}, \phi^m, f_s^i)=\Psi(W_{v'_m}, \phi^m, f_s^i)=q^{-3i-m}.\end{equation}

Next, we compute the other side local zeta integral $\Psi(W_m,
\phi^m, \tilde f_s^i)$ on the open dense subset $N(F)\setminus
N(F)TwN(F)\subset N(F)\setminus \SL_2(F)$,  where $\tilde
f_s^i=M_s(f_s^i)$. We have
\begin{align*}
\Psi(W_m, \phi^m, \tilde f_s^i)&=\int_{T\times N(F)}W_m(t(a)wn)(\omega_{\psi^{-1}}(t(a)wn)\phi^m)(1)
\tilde f_s^i(t(a)wn)|a|^{-2}dn da\\
&=\int_{T\times N_m}W_m(t(a)wn)(\omega_{\psi^{-1}}(t(a)wn)\phi^m)(1) \tilde f_s^i(t(a)wn)|a|^{-2}dn da\\
&+\int_{T\times
(N(F)-N_m)}W_m(t(a)wn)(\omega_{\psi^{-1}}(t(a)wn)\phi^m)(1) \tilde
f_s^i(t(a)wn)|a|^{-2}dn da.
\end{align*}
By Lemma 3.6, we get $W_{v_m}(t(a)wn)=W_{v_m'}(t(a)wn)$ for all
$n\in N(F)-N_m$. Thus
\begin{align*}
&\Psi(W_{v_m},\phi^m,\tilde f_s^i)-\Psi(W_{v_m'}, \phi^m, \tilde f_s^i)\\
=& \int_{T\times N_m}( W_{v_m}(t(a)wn)-W_{v_m'}(t(a)wn))(\omega_{\psi^{-1}}(t(a)wn)\phi^m)(1)
\tilde f_s^i(t(a)wn)|a|^{-2}dn da.
\end{align*}
Since $i\ge I(N_m, \eta)$, we get $$\tilde f_s^i(t(a)w
n)=\mu_{\psi^{-1}}(a)\eta^{-1}_{3/2-s}(a) q^{-3i}_F,$$ by Lemma 3.8.
On the other hand, by Lemma 3.1 and Lemma 3.9, for $n\in N_m$, we
get
$$ W_{m}(t(a)wn)=\psi(n)W_m(t(a)w), (\omega_{\psi^{-1}}(t(a)w n)\phi^m)(1)=\psi^{-1}(n)
(\omega_{\psi^{-1}}(t(a)w )\phi^m)(1). $$
Thus
\begin{align}
&\Psi(W_{v_m},\phi^m,\tilde f_s^i)-\Psi(W_{v_m'}, \phi^m, \tilde f_s^i)\\
=& q_F^{-3i+m}\int_{T}( W_{v_m}(t(a)w)-W_{v_m'}(t(a)w))(\omega_{\psi^{-1}}(w)\phi^m)(a)
\chi(a)\eta^{-1}(a)|a|^{-s} da.\nonumber
\end{align}

By (3.2), (3.3) and the local functional equation, we get
\begin{align}
&q^{-2m}(\gamma(s,\pi,\eta,\psi)-\gamma(s,\pi',\eta,\psi))\\
=&\int_{F^\times} ( W_{v_m}(t(a)w)-W_{v_m'}(t(a)w))(\omega_{\psi^{-1}}(w)\phi^m)(a)
\chi(a)\eta^{-1}(a)|a|^{-s} da.\nonumber
\end{align}
Let $k=4L$. Since $m\ge 6L>k$, by Lemma 3.1 and Lemma 3.6, we get
\begin{align*}
W_{v_m}(t(a)w)-W_{v_m'}(t(a)w)
=&\frac{1}{\vol(N_m)}\int_{N_m} (W_{v_k}(t(a)w n)-W_{v_k'}(t(a)wn) )\psi^{-1}(n)dn\\
=& \frac{1}{\vol(N_m)} \int_{N_k}(W_{v_k}(t(a)w n)-W_{v_k'}(t(a)wn) )\psi^{-1}(n)dn\\
=&\frac{\vol(N_k)}{\vol(N_m)} (W_{v_k}(t(a)w)-W_{v_k'}(t(a)w))\\
=& q^{k-m} (W_{v_k}(t(a)w)-W_{v_k'}(t(a)w)).
\end{align*}

Now we can rewrite (3.4) as
\begin{align}
&q^{-m-k}(\gamma(s,\pi,\eta,\psi)-\gamma(s,\pi',\eta,\psi))\\
=&\int_{F^\times} ( W_{v_k}(t(a)w)-W_{v_k'}(t(a)w))(\omega_{\psi^{-1}}(w)\phi^m)(a)
\chi(a)\eta^{-1}(a)|a|^{-s} da.\nonumber
\end{align}
By Lemma 3.2, if $a\notin \CP^{-6L}$, i.e., $a^2\notin \CP^{-3k}$,
we get $W_{v_k}(t(a)w)=0=W_{v_k'}(t(a)w)$. Thus the integral on the
right side in the above formula (3.5) can be taken over $\CP^{-6L}$. For
$a\in \CP^{-6L}$ and $m\ge 6L$ (as we assumed), by the calculation given in the proof of Lemma 3.9, we have
\begin{align*}
(\omega_{\psi^{-1}}(w)\phi^m)(a)
&=\gamma(\psi^{-1})\psi^{-1}(2a)\vol(\CP^m) \textrm{Char}(\CP^{-m})(a)\\
&=\gamma(\psi^{-1})\psi^{-1}(2a)q^{-m}.
\end{align*}
Plugging this into (3.5), we get
\begin{align}
&q^{-k}\gamma(\psi^{-1})^{-1}(\gamma(s,\pi,\eta,\psi)-\gamma(s,\pi',\eta,\psi))\\
=&\int_{F^\times} ( W_{v_k}(t(a)w)-W_{v_k'}(t(a)w))\psi^{-1}(2a) \chi(a)\eta^{-1}(a)|a|^{-s} da.\nonumber
\end{align}
Now we can prove our theorem. We consider (1) first. Suppose that
$\gamma(s,\pi,\eta,\psi)=\gamma(s,\pi',\eta,\psi) $ for all
quasi-characters $\eta$ of $F^\times$, we get
$$ \int_{F^\times} ( W_{v_k}(t(a)w)-W_{v_k'}(t(a)w))\psi^{-1}(2a) \chi(a)\eta^{-1}(a)|a|^{-s} da=0$$
for all quasi-characters $\eta$. By Mellin inversion, we get
$$ ( W_{v_k}(t(a)w)-W_{v_k'}(t(a)w))\psi^{-1}(2a)\equiv 0,$$
or $$ W_{v_k}(t(a)w)\equiv W_{v_k'}(t(a)w).$$ By Lemma 3.1, Lemma
3.6, Corollary 3.5 and the Iwasawa decomposition $\SL_2=B\cup BwB$,
we get
$$ W_{v_k}(g)=W_{v'_k}(g),$$
for all $g\in \SL_2(F)$. By the uniqueness of Whittaker model, we get $\pi\cong \pi'$. This proves (1).

Next, we consider (2). Let $l=l(\pi,\pi')$ be an integer such that $l\ge 6L$ and
$$W_{v_k}(t(a_0a)w)=W_{v_k}(t(a)w), \textrm{ and } W_{v_k'}(t(a_0a)w)=W_{v_k'}(t(a)w),$$
for all $a_0\in 1+\CP^l$ and all $a\in \CP^{-6L}$. Such an $l$
exists because the functions $a\mapsto W_{v_k}(t(a)w)$, $a\mapsto
W_{v_k'}(t(a)w)$ on $\CP^{-6L}\subset F^\times$ are continuous. Note
that $k=4L$ and $L$ only depends on the choice of $v$ and $v'$. On
the other hand, for $a\in \CP^{-6L}$, it is easy to see that
$$\psi^{-1}(2a_0a)=\psi^{-1}(2a), \textrm{ for all }a_0\in 1+\CP^l,$$
since $l\ge 6L$. It is also clear that $\chi(a_0a)=\chi(a)$ for all
$a_0\in 1+\CP^l$. In fact, the character $\chi$ is unramified. As we
noticed before, the integrand of the right side integral of (3.6)
has support in $\CP^{-6L}$. Let $\eta$ be a quasi-character of
$F^\times$ with $\cond(\eta)>l$, now it is clear that integral of
the right side of (3.6) vanishes, thus we get
$$\gamma(s,\pi,\eta,\psi)=\gamma(s,\pi',\eta,\psi).$$
This finishes the proof.
\end{proof}

\section{A strong multiplicity one theorem}

\subsection{Global genericness} In this subsection, we discuss the relation
between globally generic and locally generic. Let $F$ be a number
field and $\BA$ be its adele. Let $\varphi$ be a cusp form on
$\SL_2(F)\setminus \SL_2(\BA)$. Since the group $N(F)\setminus
N(\BA)$ is compact and abelian, we have the Fourier expansion
$$\varphi(g)=\sum_{\psi\in \widehat{N(F)\setminus N(\BA)}} W_\varphi^\psi(g),$$
where $$W_\varphi^\psi(g)=\int_{N(F)\setminus N(\BA)} \varphi(ng)\psi^{-1}(n)dg.$$
Since $\varphi$ is a cusp form, we get $\varphi_0\equiv 0$, thus we get
$$\varphi(g)=\sum_{\psi\in \widehat{N(F)\setminus N(\BA)} \atop \psi\ne 0}W_\varphi^\psi(g).$$
Fix a nontrivial additive character $\psi$ of $N(F)\setminus N(\BA)$, then
$$\wpair{\widehat{N(F)\setminus N(\BA)}}\setminus \{0\}=\wpair{\psi_\kappa, \kappa\in F^\times},$$
where $\psi_\kappa(a)=\psi(\kappa a), a\in \BA$. If $\kappa\in F^{\times,2}$, say $\kappa=a^2$, we have
$$W_{\varphi}^{\psi_\kappa }(g)=W_{\varphi}^\psi(t(a)g).$$
Thus we get
$$\varphi(g)=\sum_{\kappa\in F^\times/ F^{\times,2}} \sum_{a\in F^\times} W_{\varphi}^{\psi_\kappa}(t(a)g).$$

\begin{cor}
If $\varphi$ is a nonzero cusp form, there exists $\kappa\in F^\times$ such that
$$W_{\varphi}^{\psi_\kappa}\ne 0.$$
\end{cor}

Let $(\pi,V_\pi)$ be a cuspidal automorphic representation of
$\SL_2(F)\setminus \SL_2(\BA)$. We call $\pi$ is
$\psi_\kappa$-generic, if there exists $\varphi\in V_\pi$ such that
$$W_\varphi^{\psi_\kappa}\nequiv 0.$$

\begin{cor}
Let $\pi$ be a cuspidal automorphic representation of
$\SL_2(F)\setminus \SL_2(\BA)$ and $\psi$ be a nontrivial additive
character of $F\setminus \BA$. Then there exists $\kappa\in
F^\times$, such that $\pi$ is $\psi_\kappa$-generic.
\end{cor}

\begin{thm}
Let $\pi=\otimes_v\pi_v$ be an irreducible cuspidal automorphic
representation of $\SL_2(\BA)$ and $\psi=\otimes \psi_v$ be a
nontrivial additive character of $F\setminus \BA$. Then $\pi$ is
$\psi$-generic if and only if each $\pi_v$ is $\psi_v$-generic.
\end{thm}
\begin{proof}
A similar result is proved for $\RU(1,1)$ by Gelbart, Rogawski and Soudry, in Proposition 2.5, \cite{GeRS}.

One direction is trivial. Now we assume each $\pi_v$ is $\psi_v$-generic.

We assume $\pi$ is $\psi_\kappa$-generic for some $\kappa\in
F^\times$, i.e., there exists $\varphi\in V_\pi$ such that
$$W_\varphi^{\psi_\kappa}(g)=\int_{N(F)\setminus N(\BA)} \varphi(ng)\psi^{-1}_\kappa(n)dn\ne 0.$$
Then it is clear that $\pi_v$ is also $\psi_{\kappa,v}-generic$,
where $\psi_{\kappa,v}(a)=\psi_v(\kappa a)$. By the theorem in the
first section, we get $\pi_v\cong \pi_v^\kappa$.

For $\varphi\in V_\pi$, consider the function
$\varphi^\kappa(g)=\varphi(g^\kappa)$, where
$g^\kappa=\diag(\kappa,1)g \diag(\kappa^{-1},1)$. Then
\[
\int_{N(F)\setminus N(\BA)}\varphi^\kappa (ng)dn=\int_{N(F)\setminus
N(\BA)} \varphi((ng)^\kappa)dn
\]
\[=\int_{N(F)\setminus N(\BA)} \varphi(n^\kappa g^\kappa)dn=\int_{N(F)\setminus
N(\BA)}\varphi(ng^\kappa)dn=0
\]
hence $\varphi^\kappa$ is also a cusp form. Let $V_\pi^\kappa$ be
the space which consists functions of the form $\varphi^\kappa$ for
all $\varphi \in V_\pi$. Let $\pi^\kappa$ denote the cuspidal
automorphic representation of $\SL_2(\BA)$ on $V_\pi^\kappa$.

\textbf{Claim}: $(\pi^\kappa)_v=\pi_v^\kappa$.

\textit{proof of the claim:} Let $\Lambda: V_\pi\to \BC$ be a nonzero $\psi_\kappa$-Whittaker functional of $\pi$, and let $\Lambda_v$ be nonzero $(\psi_\kappa)_v$-Whittaker functional on $V_{\pi_v}$ satisfying that, if $\varphi=\otimes_v \varphi_v$ is a pure tensor, then 
\[
\Lambda(\pi(g)\varphi)=\prod_v \Lambda_v(\pi_v(g_v)\varphi_v)
\]
Note that $\Lambda$ is in fact given by 
\[
\Lambda(\varphi)=\int_{N(F)\setminus N(\BA)} \varphi(n)\psi_\kappa^{-1}(n)dn
\]
Then the $\psi_{\kappa^2}$-Whittaker functional of $\pi^\kappa$ is given by 
\[
\int_{N(F)\setminus N(\BA)} \varphi^\kappa (n)\psi^{-1}_{\kappa^2}(n)dn
\]
This means that if $W_\varphi(g)$ is a $\psi_\kappa$-Whittaker function of $\pi$, then $W_{\varphi^\kappa}(g)=W_\varphi(g^\kappa)$ is a $\psi_{\kappa^2}$-Whittaker function of $\pi^\kappa$. 

Hence with $\varphi=\otimes_v \varphi_v$ a pure tensor, we have $W_{\varphi}(g)=\prod_v W_{\varphi_v}(g_v)$, and $\{W_{\varphi_v}(g_v) \}$ is the Whittaker model of $\pi_v$. While $W_{\varphi^\kappa}(g)=W_{\varphi}(g^\kappa)= \prod_v W_{\varphi_v}(g_v^\kappa)$, and $\{W_{\varphi_v}(g_v^\kappa) \}$ is the Whittaker model of $(\pi^\kappa)_v$. 

Now $W_v(g_v)\to W_v(g_v^\kappa)$ gives an isomorphism between $\pi_v^k$ and $(\pi^\kappa)_v$, which proves the claim. 

Now let's continue the proof of the theorem, by the claim we then
have $\pi_v\cong (\pi^\kappa)_v$, or $\pi\cong \pi^\kappa$. By
multiplicity one theorem for $\SL_2$, see \cite{Ra}, we get $\pi=\pi^\kappa$. It is clear that
$\pi^\kappa$ is $\psi_{\kappa^2}$-generic, and hence $\psi$-generic.
The theorem follows.
\end{proof}

\subsection{Eisenstein series on $\widetilde \SL_2(\BA)$}
Now let $F$ be a number field, and $\BA$ be its adele ring. Let
$\widetilde \SL_2(\BA)$ be the double cover of $\SL_2(\BA)$. It is
well-known that the projection $\widetilde \SL_2(\BA)\ra \SL_2(\BA)$
factors through $\SL_2(F)$. Let $\mu_\psi$ be the genuine character of $T(F)\setminus \widetilde T(\BA)$ whose local components are $\mu_{\psi_v}$ given in $\S$2.

Let $\eta$ be a quasi-character of $F^\times \setminus \BA^\times$,
and $s\in \BC$, we consider the induced representation
$$I(s,\chi,\psi)=\Ind_{\widetilde B(\BA)}^{\widetilde \SL_2(\BA)}( \mu_\psi\eta_{s-1/2}).$$
For $f_s\in I(s,\eta,\psi)$, we consider the Eisenstein series $E(g,f_s)$ on $\widetilde \SL_2(\BA):$
$$E(g,f_s)=\sum_{B(F)\setminus \SL_2(F)} f_s(\gamma g),g\in \SL_2(\BA).$$
The above sum is absolutely convergent when $\Re(s)>>0$, and can be meromorphically continued to the whole $s$-plane.

There is an intertwining operator $M_s: I(s,\eta,\psi)\ra I(1-s, \eta^{-1},\psi):$
$$M_s(f_s)(g)=\int_{N(F)\setminus N(\BA)}f_s(wng)dn. $$
The above integral is absolutely convergent for $\Re(s)>>0$ and defines a meromorphic function of $s\in \BC$.

\begin{prop}
\begin{enumerate}
\item If $\eta^2\ne 1$, then the Eisenstein series $E(g,f_s)$ is holomorphic for all $s$. If $\eta^2=1$,
the only possible poles of $E(g,f_s)$ are at $s=0$ and $s=1$. Moreover, the order of the poles are at most 1.
\item We have the functional equation $$E(g,f_s)=E(g, M_s(f_s)), \textrm{ and } M_s(\eta)\circ M_{1-s}(\eta^{-1})=1.$$
\end{enumerate}
\end{prop}
See Proposition 6.1 of \cite{GQT} for example.

\subsection{The global zeta integral}
Let $\psi$ be a nontrivial additive character of $F\setminus \BA$.
Then there is a global Weil representation representation
$\omega_\psi$ of $\widetilde \SL_2(\BA)$ on $\CS(\BA)$. For $\phi\in
\CS(\BA)$, we consider the theta series
$$\theta_{\psi}(\phi)(g)=\sum_{x\in F}(\omega_\psi(g)\phi)(x).$$
It is well-known that $\theta_\psi$ defines an automorphic form on $\widetilde \SL_2(\BA)$.

Let $(\pi,V_\pi)$ be a $\psi$-generic cuspidal automorphic
representation of $\SL_2(\BA)$. For $\varphi\in V_\pi, \phi\in
\CS(\BA)$ and $f_s\in I(s,\eta,\psi^{-1})$, we consider the integral
\begin{equation}
Z(\varphi,\theta_{\psi^{-1}}(\phi), E(\cdot, f_s))=\int_{\SL_2(F)\setminus \SL_2(\BA)}
\varphi(g)\theta_{\psi^{-1}}(\phi)(g)E(g,f_s)dg.
\end{equation}
\begin{prop}[Theorem 4.C of \cite{GPS2}]
For $\Re(s)>>0$, the integral $Z(\varphi,\theta_{\psi^{-1}}(\phi), E(\cdot, f_s))$ is absolutely convergent, and
$$Z(\varphi, \theta_{\psi^{-1}}(\phi), E(\cdot,f_s))=\int_{N(\BA)\setminus
\SL_2(\BA)} W_{\varphi}^{\psi}(g) (\omega_{\psi^{-1}}(g))\phi(1) f_s(g)dg,$$
where $W_{\varphi}^{\psi}(g)=\int_{N(F)\setminus N(\BA)}
\varphi(ng)\psi^{-1}(n)dn$ is the $\psi$-th Whittaker coefficient of
$\varphi$.
\end{prop}

\begin{cor}
We take $\phi=\otimes_v \phi_v, f_s=\otimes f_{s,v}$ to be pure tensors. Let $S$ be a
finite set of places such that for all $v\notin S$, then $v$ is
finite and $\pi_v,\psi_v,f_{s,v}$ are unramified, then for
$\Re(s)>>0$, we have
$$Z(\varphi, \theta_{\psi^{-1}}(\phi), E(\cdot, f_s))=\prod_{v\in S}\Psi(W_v, \phi_v, f_{s,v})
\frac{L^S(s, \pi, St\otimes (\chi \eta))}{L^S(2s, \eta^2)},$$
where $\chi$ is the character of $F^\times \setminus \BA^\times$ defined by
$$\chi((a_v))=\prod_v (a_v,-1)_{F_v}, (a_v)_v \in \BA^\times.$$
Moreover, we have the following functional equation
\[
Z(\varphi, \theta_{\psi^{-1}}(\phi), E(\cdot, f_s))=Z(\varphi,
\theta_{\psi^{-1}}(\phi), E(\cdot, M_s(f_s)))
\]
\end{cor}

This follows from Proposition 4.2, the unramified calculation and
the functional equation of Eisenstein series in Proposition 4.4
directly.

\begin{cor}
\begin{enumerate}
\item The partial $L$-function $L^S(s,\pi,St\otimes \chi\eta)$ can be extended to a meromorphic of $s$.
\item If $\eta^2\ne 1$, then $L^S(s,\pi, St\otimes \chi\eta)$ is holomorphic for $\Re(s)>1/2$.
\item If $\eta^2=1$, then on the region $\Re(s)>1/2$, the only possible pole of
$L^S(s,\pi, St\otimes \chi \eta)$ is at $s=1$. Moreover, the order of the pole of
$L^S(s,\pi, St\otimes (\chi \eta))$ at $s=1$ is at most $1$.
\end{enumerate}
\end{cor}
\begin{proof}
By Proposition 4.4 and Corollary 4.6, it suffices to show that, for each place $v$, and for any fixed point $s\in \BC$ we can choose datum $(W_v,\phi_v,f_{s,v})$ such that $\Psi(W_v,\phi_v, f_{s,v})\ne 0$. If $v$ is non-archimedean, this is proved in the proof of Theorem 3.10, see Eq.(3.2).  We will prove the general case later, see Lemma 4.9.
\end{proof}

\subsection{A strong multiplicity one theorem} With the above preparation, we are now ready to prove the main
global result of this paper.

\begin{thm}
Let $\pi=\otimes \pi_v$ and $\pi'=\otimes \pi'_v$ be two irreducible cuspidal automorphic representation of
$\SL_2(\BA)$ with the same central character. Suppose that $\pi$ and $\pi'$ are both $\psi$-generic.
Let $S$ be a finite set of \textbf{finite} places such that no place in $S$ is above $2$.
If $\pi_v\cong \pi'_v$ for all $v\notin S$, then $\pi =\pi'$.
\end{thm}
\begin{proof}
For $\phi=\otimes_v \phi_v\in V_{\pi}$, consider $Z(\varphi,
\theta_{\psi^{-1}}(\phi), E(\cdot, f_s))$. By Proposition 4.5, we
have
\[
Z(\varphi, \theta_{\psi^{-1}}(\phi), E(\cdot, f_s))=\prod_v
\Psi(W_v, \phi_v, f_{s,v})
\]
\[=\prod_v\int_{N(F_v)\setminus
\SL_2(F_v)}W_v(h)(\omega_{\psi^{-1}}(h)\phi_v)(1)f_{s,v}(h)dh
\]

By functional equation in Corollary 4.6, we have
\[
\prod_v\int_{N(F_v)\setminus
\SL_2(F_v)}W_v(h)(\omega_{\psi^{-1}}(h)\phi_v)(1)f_{s,v}(h)dh
\]
\[=\prod_v\int_{N(F_v)\setminus
\SL_2(F_v)}W_v(h)(\omega_{\psi^{-1}}(h)\phi_v)(1) M_s(f_{s,v})(h)dh
\]

On archimedean places, by Lemma 4.9 below, we can choose datum so
that the archimedean local integrals are nonzero. While on the
finite places, by local functional equation discussed in section
1.4, there are local gamma factors $\gamma(s,\pi_v,\eta_v,\psi_v)$
such that
\[
\int_{N(F_v)\setminus
\SL_2(F_v)}W_v(h)(\omega_{\psi^{-1}}(h)\phi_v)(1)
M_s(f_{s,v})(h)dh=
\]
\[\gamma(s,\pi_v,\eta_v,\psi_v)
\int_{N(F_v)\setminus
\SL_2(F_v)}W_v(h)(\omega_{\psi^{-1}}(h)\phi_v)(1)f_{s,v}(h)dh
\]
Hence we have
\[
1=\prod_{v<\infty}\gamma(s,\pi_v,\eta_v,\psi_v) \prod_{v|\infty}
\frac{\Psi(W_v, \phi_v, M_s(f_{s,v}))}{\Psi(W_v, \phi_v, f_{s,v})}
\]
Similarly, we have for $\pi'$ the identity
\[
1=\prod_{v<\infty}\gamma(s,\pi'_v,\eta_v,\psi_v) \prod_{v|\infty}
\frac{\Psi(W'_v, \phi_v, M_s(f_{s,v}))}{\Psi(W'_v, \phi_v, f_{s,v})}
\]
By assumptions that $\pi_v\cong \pi'_v$ for all $v\notin S$ with $S$
doesn't have archimedean places, we then have
\[
\prod_{v\in S}\gamma(s,\pi_v,\eta_v,\psi_v)=\prod_{v\in
S}\gamma(s,\pi'_v,\eta_v,\psi_v)
\]
Fix $v_0\in S$, by Lemma 12.5 in \cite{JL}, given an arbitrary
character $\eta_{v_0}$, we can find a character $\eta$ of
$\BA^{\times}$ which restricts to $v_0$ is $\eta_{v_0}$ and has
arbitrary high conductor at other places of $S$. By Theorem 3.10
(2), we conclude that
\[
\gamma(s,\pi_{v_0},\eta_{v_0},\psi_{v_0})=\gamma(s,\pi'_{v_0},\eta_{v_0},\psi_{v_0})
\]
for all characters $\eta_{v_0}$. Thus by Theorem 3.10 (1), we
conclude that $\pi_{v_0}\cong \pi'_{v_0}$. This applies also to
other places of $S$. Thus we proved that $\pi_v\cong \pi'_v$ for all
places $v$. Now the theorem follows from the multiplicity one theorem for
$\SL_2$, \cite{Ra}.
\end{proof}
\noindent \textbf{Remark:} We expect that the restriction on the finite set $S$ in Theorem 4.8 can be removed. \\

Finally, we prove a nonvanishing result on archimedean local zeta
integrals which is used in the above proof. We formulate and prove
the result both for $p$-adic and archimedean cases simultaneously.

\begin{lem}
Let $F$ be a local field,  $\psi$ be a nontrivial additive character
of $F$, $\eta$ be a quasi-character of $F^\times$ and $\pi$ be a
$\psi$-generic representation of $\SL_2(F)$. Then there exists $W\in
\CW(\pi, \psi), \phi\in \CS(F)$ and $f_s\in \Ind_{\widetilde
B}^{\widetilde \SL_2(F)}$ such that
$$\Psi(W,\phi,f_s)=\int_{N(F)\setminus \SL_2(F)} W(h)(\omega_{\psi^{-1}}\phi)(h) f_s(h)\ne 0.$$
\end{lem}
\begin{proof}
We note that the Bruhat cell $\Omega=N(F)TwN(F)$ is open dense in
$\SL_2(F)$. Thus the above integral is reduced to
\[
\Psi(W,\phi,f_s)=\int_{TN(F)}W(\omega_2t(a)n(u)
)(\omega_{\psi^{-1}}(wt(a)n(u))\phi)(1)f_s(wt(a)n(u))\Delta(a)dadu
\]
where $\Delta(a)$ is certain Jacobian.

Use the formulas for the Weil representation $\omega_{\psi^{-1}}$,
we find
\[
(\omega_{\psi^{-1}}(\omega_2t(a)n(u))\phi)(x)=|a|^{1/2}\frac{\gamma(\psi^{-1})}
{\gamma(\psi^{-1}_a)}\int_F
\psi(ua^2y^2)\phi(ay)\psi(2xy)dy=|a|^{1/2}
\frac{\gamma(\psi^{-1})}{\gamma(\psi^{-1}_a)}\hat{\Phi}_{a,u}(x)
\]
where $\Phi_{a,u}(x)=\psi(ua^2x^2)\phi(ax)$ which is again an
Schwartz function on
$F$ and depends continuously on $a,u$.  \\

Note that $\Omega=N(F)TwN(F)$ is open in $\widetilde{\SL}_2(F)$. Now
define $f_s\in I(s,\eta,\psi^{-1})$ on the set $\{(g,1):g\in
\SL_2(F)\}$ by
\[
f_s(g)=\begin{cases} \delta(b)^{1/2}(\eta_{s-1/2}\mu_{\psi^{-1}})(b)
f_2(u)  &\text{if} \ \ g=bw n(u)\in\Omega \\ 0& \text{otherwise}
\end{cases}
\]
where $b\in B(F)=TN(F),  u\in F$, and $f_2$ is compactly supported
to be determined later. Then we extend the definition of $f_2$ to
the set $\{(g,-1):g\in \SL_2(F) \}$ to make it genuine, i.e.,
$f_s(g, -1)=-1f_s(g,1)$.

Then the integral $\Psi$ can be reduced further to
\[
\Psi(W,\phi,f_s)=\int_{TN(F)}W(wau
)|a|^{1/2}\frac{\gamma(\psi^{-1})}{\gamma(\psi^{-1}_a)}
\hat{\Phi}_{a,u}(1) \delta(a)^{1/2}(\eta_{s-1/2}\mu_{\psi^{-1}})(a)
f_2(u)\Delta(a)dadu  \hspace{1cm} \cdots (*)
\]

Case 1: $F$ is $p$-adic.   \\

Consider Howe vector $W_{v_m}$. By Corollary 3.4, taking $m$ large
enough, $W_{v_m}$ can have arbitrarily small compact open support
around $1$ when restricted to $T$. Then $W_{w.v_m}(t(a^{-1})w)$ has
small compact open support
around $a=1$.  \\

First choose $\phi$ so that $\hat{\Phi}_{a,u}(1)\neq 0$ when
$a=1,u=0$. Then choose $m$ so that
$W_{\omega_2.v_m}(wt(a))=W_{w.v_m}(t(a^{-1})w)$
has small compact support around $1$ and all other datum involving
$a$ in the integral (*) are nonzero constants. For this
$W_{w.v_m}$, consider $W_{w.v_m}(\omega_2 t(a) u)$
with $u\in N$. When $u$ is close to $1$ enough, we have
$W_{w.v_m}(w t(a) u)=W_{w.v_m}(wt(a))$
for all $a$ in that small compact support around $1$. Then take
$f_2$ with support $u$ close to $1$ satisfying the above. With these
choices of $W_{w.v_m}(g),
f_2, \phi$, the integral $(*) $ is nonzero.   \\

Case 2: $F$ is archimedean.   \\

We will concentrate on the case $F=\BR$. The case $F=\BC$ is similar
as we have the same formulas for Weil representations by Proposition
1.3 in \cite{JL}. We begin with the formulas
\[
\Psi(W,\phi,f_s)=\int_{TN(F)}W(\omega_2au
)|a|^{1/2}\frac{\gamma(\psi^{-1})}{\gamma(\psi^{-1}_a)}
\hat{\Phi}_{a,u}(1) \delta(a)^{1/2}(\eta_{s-1/2}\mu_{\psi^{-1}})(a)
f_2(u)\Delta(a)dadu  \hspace{0.5cm} \cdots (**)
\]
where $\Phi_{a,u}(x)=\psi(ua^2x^2)\phi(ax)$ is again a Schwartz
function as $\phi$ is, and it depends on $a,u$ continuously. Since
the Fourier transform is again an isometry of Schwartz function
space, we can choose $\phi$ so that the Fourier transform
$\hat{\Phi}_{a,u}(1)> 0$
when $a=1,u=0$ and depends on $a,u$ continuously.  \\

Now let $(\pi, V)$ be an irreducible generic smooth representation
of $\SL_2(\BR)$ of moderate growth. Realize $\pi$ as a quotient of
principal series smooth $I(\chi, s)$, i.e.,
\[
0\to V'\to I(\chi,s)\to V\to 0
\]
Let $\lambda:V\to \BC$ be the unique nonzero continuous Whittaker
functional on $V$, then the composition
\[\Lambda:
 I(\chi,s)\to V\xrightarrow{\lambda} \BC
\]
gives the unique nonzero continuous Whittaker functional on
$I(\chi,s)$ up to a scalar. It follows that the two spaces $\{
\lambda(\pi(g)v):g\in \SL_2(F), v\in V  \}$ and $\{\Lambda(R(g).f):
g\in \SL_2(F), f\in I(\chi,s) \}$ are the same, although the first
is the Whittaker model of $\pi$, while the later may not be a
Whittaker model of $I(\chi,s)$.

The Whittaker functional on $I(\chi,s)$ is given by the following
\[
\Lambda(f)=\int_{N(F)} f(\omega_2u)\psi^{-1}(u)du
\]
when $s$ is in some right half plane, and its continuation gives
Whittaker functional for all $I(\chi,s)$. Also when $f$ has support
inside $\Omega=N(F)T\omega_2N(F)$, the above integral always
converges for any $s$, and gives the Whittaker functional.

Now for such $f$, one compute that for $a=t(a)\in T$
\[
\Lambda(I(a).f)=\int_{N(F)} f(\omega_2u
a)\psi^{-1}(u)du=\chi'(a)\int_F f(\omega_2u)\psi^{-1}(a^2u)du
\]
\[
=\chi'(a)\int_F f_1(u)\psi^{-1}(a^2u)du=\chi'(a)\hat{f}_1(a^2)
\]
where $f_1$ is the restriction of $f$ to $\omega_2N$ which can be
chosen to be a Schwartz function, $\hat{f}_1$ is its Fourier
transform, and $\chi'$ is certain character. Again, as Fourier
transform gives an isometry of Schwartz functions, we can always
choose $f$ so that its Whittaker function $W_f(a)$ has arbitrary
small compact support around $1$.  By a right translation by
$\omega_2$, we shows that one can always choose $f$ so that
$W_{\omega_2.f}(a\omega_2)$ has small compact support around $1$.

In order to prove the proposition, note that we have chosen $\Phi$.
Let
\[
R(a,u)=|a|^{1/2}\frac{\gamma(\psi^{-1})}{\gamma(\psi^{-1}_a)}\hat{\Phi}_{a,u}(1)
\delta(a)^{1/2} (\eta_{s-1/2}\mu_{\psi^{-1}})(a) \Delta(a)
\]
Then $R(a,u)$ is a continuous function of $a,u$ and $R(1,0)\neq 0$.
This means that there exist neighborhoods $U_1$ of $a=1$ and $U_2$
of $u=0$, such that $R(a,u)>R(1,0)/2>0$ for all $a\in U_1, u\in
U_2$.

Now choose $f$ so that $W_{\omega_2.f}(a\omega_2)$ has small compact
support in a neighborhood $V_1$ of $1$ with $V_1\subset U_1$, and
$W_{\omega_2.f}(\omega_2)>0$ . For this Whittaker function, since
$W_{\omega_2.f}(a\omega_2)u$ is continuous on $u$, we can choose
$f_2$ so that it is positively supported in a neighborhood $V_2$ of
$0$ such that  \\
 (1). $V_2\subset U_2$; \\
 (2). $W_{\omega_2.f}(a\omega_2u)>W_{\omega_2.f}(\omega_2)/2>0$ for all $u\in V_2$. \\

 Then $(**)$ becomes
 \[
 \int W_{\omega_2.f}(a\omega_2u)R(a,u)f_2(u)dadu> \frac{W_{\omega_2.f}(\omega_2)}{2}
 \frac{R(1,0)}{2}\int_{V_1}\int_{V_2} f_2(u)dadu>0
 \]
which proves the non vanishing.

\end{proof}

\end{document}